\documentclass[12pt,a4paper]{article}
\usepackage[active]{srcltx}
\usepackage{amsfonts}
\usepackage{verbatim}
\usepackage{amsmath}
\usepackage{amsbsy}
\usepackage{amsxtra}
\usepackage{latexsym}
\usepackage{amssymb}
\usepackage{url}
\usepackage{cite}
\usepackage{pstricks,pst-node}
\usepackage{enumerate}
\usepackage{mathrsfs}
\usepackage[margin=1.8cm,nohead]{geometry}

\makeatletter
\g@addto@macro\thesection.
\makeatother

\newtheorem{theorem}{Theorem}

\newcommand{\lf}{\left}
\newcommand{\rg}{\right}

\newcommand{\R}{\mathbb R}
\newcommand{\N}{\mathbb N}
\newcommand{\bs}{\begin{smallmatrix}}
\newcommand{\es}{\end{smallmatrix}}

\newcommand{\f}{\frac}

\newenvironment{proof}{{\noindent\bf Proof.}}{\hfill$\Box$\\}

\begin{document}

\title{The Banach fixed point principle viewed as a monotone convergence with respect to the Lorentz cone in Euclidean spaces
}
\author{S. Z. N\'emeth\thanks{School of Mathematics, University of Birmingham, Watson Building, Edgbaston, Birmingham B15 2TT, United Kingdom, 
email: {\tt s.nemeth@bham.ac.uk}}
}
\date{}
\maketitle

\begin{abstract}
	Who would have thought that a purely metrical result such as the Banach fixed point principle can be viewed as a monotone convergence? 
	The aim of this paper is to both shock and entertain the reader by presenting a surprising connection between the the Banach fixed 
	point principle and the ordering structure of the Euclidean space ordered by the Lorentz cone. The Lorentz cone is already a very
	important concept for theoretical physicists and recently for the optimization community as well, but probably very little (if any) is 
	known about its connection with fixed point principles or ordered vector spaces. In more technical terms this is what we are going to
	show: We augment the dimension of the Euclidean space by one and the Picard iteration of a contraction by a simple iteration on the real 
	line such that the resulting iteration becomes monotone increasing and bounded with respect to the order defined by the Lorentz cone of 
	the augmented space.  
\end{abstract}

\section{Introduction}
Probably most of the readers would think that the Banach fixed principle is purely metrical in nature and has nothing to do with the ordering of
vectors. To the contrary, we will exhibit an unexpected connection between the Banach fixed point principle and the Euclidean space ordered by a
Lorentz cone, cone which is already well known to theoretical physicists and relatively recently to the optimization community as well (via
second order cone programming), but little about this cone (if any) has been told related to fixed point theorems or vector orders. How will we 
do this? ``To get to the icing 
on the cake'', we need some terminology and reminders first: Let $m$ be a positive integer, $p=m+1$ and $\|\cdot\|$ the Euclidean norm in $\R^p$.
A \emph{cone} $K\subset\R^p$ is a closed set such that $tK+sK\subset K$ and $K\cap(-K)=\{0\}$ for any $t,s\ge0$ \cite{Boyd2004}.  The 
\emph{order} $\le_K$ induced by the cone $K$ is defined by the equvalence $x\le_K y\iff y-x\in K$. This order is reflexive, transitive and 
antisymmetric. Moreover, $\le_K$ is \emph{compatible with the linear structure} of $\R^p$, that is, $x\le_K y$ implies that $\mu x+z\le\mu y+z$, 
for any $\mu\ge0$ and any $z\in\R^p$. The pair $(\R^p,K)$ is called an \emph{ordered Euclidean space} and $K$ the \emph{positive cone} of 
$(\R^p,K)$. A sequence in $\R^p$ will be called \emph{$K$-increasing ($K$-bounded from above)} if it is increasing (bounded from above) with 
respect to $\le_K$. A \emph{lower $K$-bound} of a set in $\R^p$ is a lower bound of the set with respect to $\le_K$. It is known that any
$K$-increasing and $K$-bounded sequence from above in $\R^p$ is convergent \cite{McArthur1970}. Let $0<\lambda<1$. A 
\emph{$\lambda$-contraction} $F:\R^p\to\R^p$ is a mapping $F:\R^p\to\R^p$ such that $\|F(x)-F(y)\|\le\lambda\|x-y\|$ for any $x,y\in\R^p$. It is 
easy to see that a contraction can have at most one fixed point. The \emph{Banach fixed point theorem (or principle)} 
(see \cite{Banach1922,Granas2003}) states that any contraction has a unique fixed point and the \emph{Picard iteration} $x^{n+1}=F(x^n)$ from any
starting point is convergent to this fixed point. And now the ``ice on the cake'': Define the sequence $t^n$ by $t^0=0$ and 
$t^{n+1}=\lambda t^n+\|x^1-x^0\|$ such that $t^{n+1}-t^n$ to be at least $\|x^{n+1}-x^{n}\|$, or equivalently $(x^n,t^n)$ to be monotone with 
respect to the \emph{Lorentz cone} (see Example 2.3 on page 31 of \cite{Boyd2004}) $L$ in $\R^m\times\R$ defined by 
$L=\{(x,t)\in\R^m\times\R:t\ge\|x\|\}$. The construction of $t^n$ also provides the $L$-boundedness from above of $(x^n,t^n)$, which therefore it 
is convergent. More details will follow in the next section.
\section{The main result}

\begin{theorem}
	Let $m$ be a positive integer, $L$ be the Lorentz cone in $\R^m\times\R$, $0<\lambda<1$, $x^0\in\R^m$ and $f:\R^m\to\R^m$ be a 
	$\lambda$-contraction. Consider the Picard iteration
	\begin{equation}\label{ex}
		x^{n+1}=f(x^n),\textrm{ }n\in\N,
	\end{equation}
	starting from $x^0$, the iteration 
	\begin{equation}\label{et}
		t^{n+1}=\lambda t^n+\|x^1-x^0\|
	\end{equation}
	starting from $t^0=0$ and the nonempty set
	\begin{equation}\label{eo}
		\Omega=\lf\{(x,t)\in\R^m\times\R:t\ge\|x-x^0\|,
		t\ge\f{\|x^1-x^0\|+\|f(x)-x\|}{1-\lambda}\rg\}.
	\end{equation}
	Then, the sequence $(x^n,t^n)$ is $L$-increasing and $L$-bounded frome above by any element
	of $\Omega$, hence it is convergent. Its limit $(x^*,t^*)$ is a lower $L$-bound of $\Omega$ with $x^*$ the unique fixed point of $f$. 
\end{theorem}
\begin{proof}
	First, we will show by induction that \[(x^n,t^n)\le_L (x^{n+1},t^{n+1}),\] for any $n\in\N$.
	Indeed, we have $t^1-t^0=\|x^1-x^0\|$ and hence $(x^0,t^0)\le_L(x^1,t^1)$. Hence, the 
	statement is true for $n=0$. Now, suppose that the statement is true for $n$, that is,
	$(x^n,t^n)\le_L (x^{n+1},t^{n+1})$. Hence, since $f$ is a $\lambda$-contraction, we get
	\begin{equation}\label{eint}
		\|x^{n+2}-x^{n+1}\|\le\lambda\|x^{n+1}-x^n\|\le\lambda(t^{n+1}-t^n), 
	\end{equation}
	where the second inequality follows from the induction hypothesis. From \eqref{et} and
	\eqref{eint}, we get \[t^{n+2}-t^{n+1}=\lambda(t^{n+1}-t^n)\ge\|x^{n+2}-x^{n+1}\|,\]
	or equivalently $(x^{n+1},t^{n+1})\le_L (x^{n+2},t^{n+2})$, that is, the statement is true
	for $n+1$. Hence, the statement is true for any $n\in\N$.

	Next, consider an arbitrary element $(x,t)\in\Omega$. We will show by induction that
	\[(x^n,t^n)\le_L(x,t).\] From \eqref{eo} we get $t\ge\|x-x^0\|$, which implies 
	$(x^0,t^0)=(x^0,0)\le_L(x,t)$. Hence, the statement is true for $n=0$. Now, suppose that the 
	statement is true for $n$, that is, $(x^n,t^n)\le_L(x,t)$, or equivalently 
	\begin{equation}\label{eih}
		t-t^n\ge\|x-x^n\|.
	\end{equation}
	Then, since $f$ is a $\lambda$-contraction, \eqref{et}, \eqref{eo}, \eqref{eih} and the 
	triangle inequality imply
	\begin{gather*}
		t-t^{n+1}=(1-\lambda)t+\lambda (t-t^n)-\|x^1-x^0\|\ge(1-\lambda)t
		+\lambda\|x-x^n\|-\|x^1-x^0\|\\\ge(1-\lambda)t+\|f(x)-f(x^n)\|-\|x^1-x^0\|
		\ge\|x-f(x)\|+\|f(x)-x^{n+1}\|\\\ge\|x-f(x)+f(x)-x^{n+1}\|=\|x-x^{n+1}\|,
	\end{gather*}
	that is the statement is true for $n+1$. Hence, the statement is true for any $n\in\N$.
	Taking the limit in \eqref{eih} as $n\to\infty$, it follows that $t-t^*\ge\|x-x^*\|$, or
	equivalently $(x^*,t^*)\le_L(x,t)$. Since $(x,t)$ is an arbitrary element of $\Omega$, it
	follows that $(x^*,t^*)$ is a lower $L$-bound of $\Omega$.
\end{proof}

\bibliographystyle{abbrv}
\bibliography{banach}

\end{document}